\documentclass[10pt]{amsart}

\usepackage{hyperref}
\hypersetup{colorlinks=true, urlcolor=blue, citecolor=blue, linkcolor=blue}
\usepackage{enumerate}
\usepackage{amsmath,amsthm, amssymb}

\usepackage{graphicx,mathrsfs,tikz,fancyhdr}

\usepackage{thmtools}
\usepackage{nameref}
\usepackage[backend=bibtex]{biblatex}

\bibliography{Masseybib.bib}

\newcommand{\coker}{{\operatorname{coker}}}

\DeclareMathOperator{\Id}{id}
\DeclareMathOperator{\Gr}{Gr}
\newcommand{\U}{{\mathcal U}}
\newcommand{\0}{{\mathbf 0}}
\newcommand{\C}{{\mathbb C}}

\newcommand{\Z}{{\mathbb Z}}

\newcommand{\Q}{{\mathbb Q}}

\renewcommand{\P}{{\mathbb P}}

\newcommand{\HH}{{\mathbb H}}

\newcommand{\hyp}{{\mathbb H}}
\newcommand{\supp}{\operatorname{supp}}

\newcommand{\im}{\mathop{\rm im}\nolimits}

\newcommand{\Idot}{\mathbf I^\bullet}

\newcommand{\Ndot}{\mathbf  N^\bullet}

\newtheorem{defn0}{Definition}[section]
\newtheorem{prop0}[defn0]{Proposition}
\newtheorem{conj0}[defn0]{Conjecture}
\newtheorem{thm0}[defn0]{Theorem}
\newtheorem{lem0}[defn0]{Lemma}
\newtheorem{corollary0}[defn0]{Corollary}
\newtheorem{example0}[defn0]{Example}
\newtheorem{remark0}[defn0]{Remark}
\newtheorem{question0}[defn0]{Question}

\newenvironment{prop}{\begin{prop0}}{\end{prop0}}

\newenvironment{thm}{\begin{thm0}}{\end{thm0}}
\newenvironment{lem}{\begin{lem0}}{\end{lem0}}
\newenvironment{cor}{\begin{corollary0}}{\end{corollary0}}

\newenvironment{rem}{\begin{remark0}\rm}{\end{remark0}}

\newtheorem{main}{Main Theorem}

\newcommand{\propref}[1]{Proposition~\ref{#1}}
\newcommand{\thmref}[1]{Theorem~\ref{#1}}
\newcommand{\lemref}[1]{Lemma~\ref{#1}}

\newcommand{\secref}[1]{Section~\ref{#1}}
\newcommand{\remref}[1]{Remark~\ref{#1}}

\title{Rational Homology Manifolds and \\
Hypersurface Normalizations}
\author{Brian Hepler}

\begin{document}

\maketitle 
\begin{abstract}
We prove a criterion for determining whether the normalization of a complex analytic space on which the shifted constant sheaf is perverse is a rational homology manifold, using a perverse sheaf known as the multiple-point complex. This perverse sheaf is naturally associated to any ``parameterized space", and has several interesting connections with the Milnor monodromy and mixed Hodge Modules. 
\end{abstract}

\section{Introduction}\label{sec:intro}
Let $\U$ be an open neighborhood of the origin in $\C^N$, let $X \subseteq \U$ be a complex analytic space containing $\0$ of pure dimension $n$,  on which the (shifted) constant sheaf $\Q_X^\bullet[n]$ is perverse (e.g., if $X$ is a local complete intersection), and let $\pi : Y \to X$ be the normalization of $X$. 

There is then a surjection of perverse sheaves $\Q_X^\bullet[n] \to \Idot_X \to 0$, where $\Idot_X$ is the intersection cohomology complex on $X$ with constant $\Q$ coefficients. 

\begin{rem}
It is a classic result (see, e.g., \cite{inthom2} page 111) that there always exists a morphism $\Q_X^\bullet[n] \to \Idot_X$ in the derived category $D_c^b(X)$ (where $\Idot_X$ has constant $\Q$ coefficients). In our situation, $\Q_X^\bullet[n]$ is a perverse sheaf on $X$, so this descends to a morphism of perverse sheaves. Since we are working with field coefficients, $\Idot_X$ is a semi-simple object in the category of perverse sheaves on $X$, so one easily concludes that the cokernel of this morphism must be zero, since the morphism $\Q_X^\bullet[n] \to \Idot_X$ is an isomorphism when restricted to the smooth part of $X$.

It is worth noting that this morphism also exists with $\Z$ coefficients (and, for a local complete intersection, $\Z_X^\bullet[n]$ is still perverse), and the morphism is still surjective, but we can no longer use the fact that $\Idot_X$ is a semi-simple object. Instead, we use that $\Idot_X$ is also the intermediate extension of the local system $\Z_{X \backslash \Sigma X}^\bullet[n]$ to all of $X$ (where $\Sigma X$ denotes the singular locus of $X$), and therefore has no perverse quotient objects with support contained in $\Sigma X$. Since $\Z_X^\bullet[n] \to \Idot_X$ is an isomorphism when restricted to $X \backslash \Sigma X$, it follows that the cokernel of this morphism is zero.

\end{rem}
Since the category of perverse sheaves on $X$ is Abelian, there is a perverse sheaf $\Ndot_X$ on $X$ such that 
$$
0 \to \Ndot_X \to \Q_X^\bullet[n] \to \Idot_X \to 0 \quad (\dag)
$$
is a short exact sequence of perverse sheaves. 

Thus, if $\Idot_Y$ is intersection cohomology on $Y$ with constant $\Q$ coefficients, we have $\pi_*\Idot_Y \cong \Idot_X$ ($\pi$ is a small resolution, in the sense of Goresky and Macpherson \cite{inthom2}), and we obtain the short exact sequence of perverse sheaves
$$
0 \to \Ndot_X \to \Q_X^\bullet[n] \to \pi_*\Idot_Y \to 0 
$$
on $X$. We refer to this exact sequence as the \textbf{fundamental short exact sequence of the normalization}. This short exact sequence, and the perverse sheaf $\Ndot_X$ in particular, have been examined recently in several papers by the author and D. Massey in the case where the normalization $Y$ is smooth (\cite{hepmasparam}, \cite{hepdefhyper}), where $\Ndot_X$ is called the \textbf{multiple-point complex} of the normalization (see \secref{sec:main}). 

Disregarding the normalization, if one just examines the short exact sequence $(\dag)$, D.Massey has recently shown in \cite{comparison} that, in the case where $X = V(f)$ is a hypersurface,
$$
\Ndot_X \cong \ker \{ \Id -\widetilde T_f\},
$$
where $\widetilde T_f$ is the monodromy action on the vanishing cycles $\phi_f[-1]\Q_\U^\bullet[n+1]$, and the kernel takes place in the category of perverse sheaves on $X = V(f)$. In this context, Massey refers to $\Ndot_X$ as the \textbf{comparison complex} on $X$. It also seems that one may obtain this result in the algebraic setting (with $\Q$ coefficients) using the language of mixed Hodge modules. 

Looking at $(\dag)$, one notices immediately that $\Q_X^\bullet[n] \cong \Idot_X$ if and only if $\Ndot_X = 0$; that is, the LCI $X$ is a rational homology manifold (or, a \textbf{$\Q$-homology manifold}) precisely when the complex $\Ndot_X$ vanishes (for this criterion, see for example \cite{BorhoMac}, \cite{ICmon}). We will recall $\Q$-homology manifolds and their properties in \secref{sec:main}. It is then natural to ask that, given the normalization $Y$ of $X$ and the resulting fundamental short exact sequence, is there a similar result relating $\Ndot_X$ to whether or not $Y$ is a $\Q$-homology manifold? 

We answer this question in our main result:
\begin{main}[\thmref{thm:main}]
$Y$ is a $\Q$-homology manifold if and only if $\Ndot_X$ has stalk cohomology concentrated in degree $-n+1$; i.e., for all $p \in X$, $H^k(\Ndot_X)_p$ is non-zero only possibly when $k=-n+1$.
\end{main}

In general, it is quite difficult to compute these stalk cohomology groups, even in the ``next simplest" case where the normalization of a hypersurface has an isolated singularity, e.g., the normalization of a surface with a curve singularity, which we will work out in detail in \secref{sec:examples}.

\begin{rem}\label{rem:saito}
M. Saito has recently drawn interesting connections with the multiple-point complex $\Ndot_X$ to the setting of mixed Hodge modules in a recent preprint \cite{msaitow0}. In particular, Saito shows, for an arbitrary reduced complex algebraic variety $X$ of pure dimension $n$, that the weight zero part of the cohomology group $H^1(X;\Q)$ is given by
$$
W_0 H^1(X;\Q) \cong \coker \{ H^0(Y;\Q) \to H^0(X;\mathcal{F}_X) \},
$$
where $\pi :Y \to X$ is the normalization of $X$, and $\mathcal{F}_X$ is a certain constructible sheaf on $X$, given by the cokernel of the natural morphism of sheaves $\Q_X \to \pi_* \Q_Y$. The algebraic setting is necessary here, in order to endow $H^0(X;\mathcal{F}_X)$ with a mixed Hodge structure, and for working in the derived category of mixed Hodge modules. 

This constructible sheaf $\mathcal{F}_X$ is none other than the cohomology sheaf $H^{-n+1}(\Ndot_X)$; this follows immediately from taking the long exact sequence in cohomology of the fundamental short exact sequence of the normalization. If, as in Saito's case, the sheaf $\Q_X^\bullet[n]$ is not perverse, one can obtain this isomorphism from the distinguished triangle
$$
\mathcal{F}_X[n] \to \Ndot_X[1] \to \pi_*\Ndot_Y[1] \overset{+1}{\to}
$$
obtained via the octahedral axiom in the derived category $D_c^b(X)$, together with the fact that $Y$ is normal. 

Consequently, we can interpret Saito's result as an isomorphism
$$
W_0 H^1(X;\Q) \cong \coker \{ H^0(Y;\Q) \to \hyp^{-n+1}(X;\Ndot_X) \},
$$
since $H^0(X;H^{-n+1}(\Ndot_X)) \cong \hyp^{-n+1}(X;\Ndot_X)$.  It would seem to be an interesting question in the local analytic case to relate this result with the isomorphism $\Ndot_X \cong W_{n-1} \Q_X^\bullet[n]$ obtained in \remref{rem:semisimple}, where $\Ndot_X$ is endowed with the natural structure of a mixed Hodge module on $X$.


\end{rem}

We would like to thank the Referee for suggesting the content of \remref{rem:semisimple}, and J\"{o}rg Sch\"{u}rmann for many helpful discussions regarding mixed Hodge modules in general and \remref{rem:saito} and \remref{rem:semisimple} in particular.

\section{Main result}\label{sec:main}

Before we prove our main result, we first recall a theorem of Borho and MacPherson \cite{BorhoMac} giving us several equivalent characterizations of rational homology manifolds:

\begin{thm}([B-M])\label{thm:ICman}
The following are equivalent:
\begin{enumerate}

\item $X$ is a $\Q$-homology manifold (i.e., $\Idot_X \cong \Q_X^\bullet[n]$);

\smallskip

\item $\mathcal{D}\left ( \Q_X^\bullet[n] \right ) \cong \Q_X^\bullet[n]$, where $\mathcal{D}$ is the Verdier duality functor;

\smallskip

\item For all $p \in X$, for all $k$, $H^k(X,X\backslash \{p\};\Q) = 0$ unless $k = 2n$, and $H^{2n}(X,X\backslash \{p\};\Q) \cong \Q$.

\end{enumerate}
\end{thm}

The proof of \thmref{thm:main} relies on the following well-known lemma.
\begin{lem}\label{lem:irredcomps}
Let $X$ be a complex analytic space of pure dimension $n$. Then, for $p \in X$, the rank of $H^{-n}(\Idot_X)_p$ is equal to the number of irreducible components of $X$ at $p$.
\end{lem}
\begin{proof}
This result is well-known to experts, see e.g. Theorem 1G (pg. 74) of \cite{whitneybook}, or Theorem 4 (pg. 217) \cite{lojbookintro}
\end{proof}

\smallskip

Note that taking stalk cohomology at $p \in X$ of the fundamental short exact sequence yields the short exact sequence
$$
0 \to \Q \to H^{-n}(\pi_*\Idot_Y)_p \to H^{-n+1}(\Ndot_X)_p \to 0,
$$
and isomorphisms $H^k(\pi_*\Idot_Y)_p \cong H^{k+1}(\Ndot_X)_p$ for $-n+1 \leq k \leq -1$. With this in mind, we claim that:

\smallskip

\begin{thm}\label{thm:main}
$Y$ is a $\Q$-homology manifold if and only if $\Ndot_X$ has stalk cohomology concentrated in degree $-n+1$.
\end{thm}

\begin{proof}

($\Longrightarrow$) Suppose that $Y$ is a $\Q$-homology manifold, and let $p \in X$ be arbitrary.  Since $Y$ is a $\Q$-homology manifold, $\Q_Y[n] \cong \Idot_Y$ in $D_c^b(Y)$, from which it follows $H^k(\Ndot_X)_p = 0$ for $k \neq -n+1$ by the above isomorphisms.

\bigskip

($\Longleftarrow$) Suppose that, for all $p \in X$,  $H^k(\Ndot_X)_p \neq 0$ only possibly when $k =-n+1$. We wish to show that the natural morphism $\Q_Y[n] \to \Idot_Y$ is an isomorphism in $D_c^b(Y)$. 

There is still the short exact sequence
$$
0 \to \Q \to H^{-n}(\pi_*\Idot_Y)_p \to H^{-n+1}(\Ndot_X)_p \to 0
$$
and $H^k(\pi_*\Idot_Y)_p = 0$ for $k \neq -n$, since $H^k(\pi_*\Idot_Y)_p \cong H^{k+1}(\Ndot_X)_p$ for all $p \in X$ and $-n+1 \leq k \leq -1$. In degree $-n$, we have
$$
H^{-n}(\pi_*\Idot_Y)_p \cong \bigoplus_{q \in \pi^{-1}(p)} H^{-n}(\Idot_Y)_q.
$$
This then implies that, for all $q \in Y$, $H^k(\Idot_Y)_q = 0$ for $k \neq -n$. Our goal is to calculate this stalk cohomology in degree $-n$. Since $Y$ is normal, and thus locally irreducible, it follows by \lemref{lem:irredcomps} that $H^{-n}(\Idot_Y)_q \cong \Q$ for all $q \in Y$. 
\medskip

Finally, we claim that the natural morphism $\Q_Y^\bullet[n] \to \Idot_Y$ is an isomorphism in $D_c^b(Y)$. In stalk cohomology at any point $q \in Y$, both $H^k(\Q_Y^\bullet[n])_q$ and $H^k(\Idot_Y)_q$ are non-zero only in degree $k=-n$, with stalks isomorphic to $\Q$.  Consequently, the natural morphism is an isomorphism in $D_c^b(Y)$ provided that the morphism 
$$
\Q \cong H^{-n}(\Q_Y^\bullet[n])_q \to H^{-n}(\Idot_Y)_q \cong \Q
$$
is not the zero morphism. But this is just the ``diagonal'' morphism from a single copy of $\Q$ to the number of connected components of $Y \backslash \{p\}$, which is clearly non-zero.  Thus, $Y$ is a $\Q$-homology manifold. 

\end{proof}

\bigskip

\begin{cor}
Suppose that $\Ndot_X$ has stalk cohomology concentrated in degree $   -n+1$. Then, for all $p \in X$, if $j_p : \{p\} \hookrightarrow X$ is the inclusion map, we have
$$
H^k(j_p^!\Ndot_X) \cong \left \{ \begin{matrix} \widetilde H^{n+k-1}(K_{X,p};\Q), && \text{ for $0\leq k \leq n-1$}; \\ 0, && \text{ else. } \end{matrix} \right . ,
$$
where $K_{X,p}$ denotes the real link of $X$ at $p$, i.e., the intersection of $X$ with a sphere of sufficiently small radius, centered at $p$.
\end{cor}

This follows by applying $j_p^!$ to the fundamental short exact sequence of the normalization, and taking stalk cohomology. 

\bigskip


When the normalization $Y \overset{\pi}{\to} X$ is a $\Q$-homology manifold, the short exact sequence 
$$
0 \to \Q \to H^{-n}(\pi_*\Idot_Y)_p \to H^{-n+1}(\Ndot_X)_p \to 0
$$
allows us to identify, given \lemref{lem:irredcomps}, that
$$
m(p):=\dim_\Q H^{-n+1}(\Ndot_X)_p = |\pi^{-1}(p)|-1.
$$
Consequently, we conclude that the support of $\Ndot_X$ is none other than the \textbf{image multiple-point set} of the morphism $\pi$, which we denote by $D$; precisely, we have
$$
D := \overline{ \{ p \in X \, | \, |\pi^{-1}(p)| > 1 \}}.
$$
For this reason, we have referred to the perverse sheaf $\Ndot_X$ as the \textbf{multiple-point complex} of $X$ (or, of the morphism $\pi$, as we do in \cite{hepdefhyper} and \cite{hepmasparam}). It is immediate from the fundamental short exact sequence that one always has the inclusion $D \subseteq \Sigma X$.

\smallskip

In such cases (see \secref{sec:examples}), it is useful to partition $X$ into subsets $X_k = m^{-1}(k)$ for $k \geq 1$; clearly, one has 
$$
D = \overline{\bigcup_{k >1} X_k}.
$$
Finally, since $D$ is the support of a perverse sheaf which, on an open dense subset of $D$, has non-zero stalk cohomology only in degree $-n+1$, it follows that $D$ is purely $(n-1)$-dimensional. 

\begin{rem}\label{rem:semisimple}

\medskip

When $Y$ is a $\Q$-homology manifold, in fact, both $\Idot_X$ and $\Ndot_X$ are just sheaves (up to a shift); moreover, the short exact sequence of perverse sheaves
$$
0 \to \Ndot_X \to \Q_X^\bullet[n] \to \Idot_X \to 0
$$
can be rewritten as a short exact sequence of (constructible) sheaves
$$
0 \to \Q_X \to \Idot_X[-n] \to \Ndot_X[1-n] \to 0.
$$
We then ask, is it ever the case that $\Ndot_X$ is a semi-simple perverse sheaf, so that $\Q_X^\bullet[n]$ is an extension of semi-simples? One can find a Whitney stratification $\mathfrak{S}$ of $X$ for which the sets $X_k$ are finite unions of strata for all $k$. Then, for each stratum $S \subset X_k$, the monodromy of the local system ${\Ndot_X}_{|_S}$ is determined by the monodromy of the set $\pi^{-1}(p)$ for $p \in S$; since this is a finite set with $k$ elements, it follows immediately that ${\Ndot_X}_{|_S}$ is semi-simple as a local system on $S$ (since the monodromy action is semi-simple). 
\medskip

Since ${\Ndot_X}_{|_S}$ is semi-simple as a local system for any stratum $S \subset X_k$, is $\Ndot_X$ semi-simple as a perverse sheaf? If one has a Whitney stratification of $X$ for which the sets $X_k$ are finite unions of strata, and for which the subset $D = \supp \Ndot_X$ is a union of closed strata, then the above argument demonstrates (together with \cite{msaitow0} Section 2.4) that $\Ndot_X$ is semi-simple as a perverse sheaf. In general, however, this fails to be the case (see \secref{sec:examples}). 
 
 \bigskip
 
More generally, when $\Q_X^\bullet[n]$ is a perverse sheaf, one may use the general machinery of M. Saito (see \cite{mixedhodgemod}, page 325 (4.5.9)) to obtain an isomorphism of perverse sheaves
$$
\Gr_n^W \Q_X^\bullet[n] \overset{\thicksim}{\to} \Idot_X.
$$
underlying the corresponding isomorphism of mixed Hodge modules. Since $\dim_\0 X = n$, the induced weight filtration on $\Q_X^\bullet[n]$ terminates after degree $n$, so that $W_n \Q_X^\bullet[n] \cong \Q_X^\bullet[n]$. Consequently, the above isomorphism yields a short exact sequence
$$
0 \to W_{n-1} \Q_X^\bullet[n] \to \Q_X^\bullet[n] \to \Idot_X \to 0
$$
of perverse sheaves on $X$, implying $\Ndot_X \cong W_{n-1} \Q_X^\bullet[n]$. From this identification, it follows that $\Ndot_X$ is semi-simple as a perverse sheaf provided that the weight filtration $W_i \Q_X^\bullet[n]$ of $W_{n-1} \Q_X^\bullet[n] \cong \Ndot_X$ for $i < n$ is concentrated in one degree $k < n$, i.e., $W_i \Q_X^\bullet[n] = 0$ for $i < k$ and $W_i \Q_X^\bullet[n] \cong W_k \Q_X^\bullet[n]$ for $k < i < n$. Then, $\Ndot_X \cong \Gr_k^W \Q_X^\bullet[n]$ underlies a pure polarizable Hodge module, which is therefore by construction a semi-simple perverse sheaf.  

We anticipate that the reverse implication will be more difficult, and be outside the scope of this paper.

\end{rem}

\section{Interpretation in terms of Comparison Complex}\label{sec:compcomp}

Recall that, by D. Massey, if $X = V(f)$ is a hypersurface, $\Ndot_X = \ker \{ \Id - \widetilde T_f\}$ is the perverse eigenspace of the eigenvalue 1 of the monodromy action on $\phi_f[-1]\Q_\U^\bullet[n+1]$, where $\U$ is an open neighborhood of the origin in $\C^{n+1}$. 

Since the content of this paper is interesting only in the case where $\dim_\0 \Sigma f = n-1$ (otherwise, $X$ is its own normalization), we will assume throughout that this is the case; consequently, the stalk cohomology $H^k(\phi_f[-1]\Q_\U^\bullet[n+1])_p$ is possibly non-zero only for $-n+1 \leq k \leq 0$. 

In general, it is \textbf{not the case} that, given a morphism of perverse sheaves, the cohomology of the stalk of the kernel of $G$ is isomorphic to the kernel  of the cohomology on the stalks; that is, there may exist points $p \in \Sigma f$ such that 
$$
H^k(\ker \{\Id - \overset{\thicksim}{T_f} \})_p \ncong \ker \{ \Id-  \widetilde T_{f,p}^k\}.
$$
However, this isomorphism \textbf{does hold} in degree $-n+1$ for all $p \in \Sigma f$ (See Lemma 5.1 of \cite{comparison}):

\begin{prop}\label{prop:identities}
Let $\pi : Y \to V(f)$ be the normalization of $V(f)$, and suppose $Y$ is a $\Q$-homology manifold. Then, the following isomorphisms hold for all $p \in \Sigma f$:
\begin{align*}
H^k(\ker \{ \Id - \overset{\thicksim}{T_f} \})_p &\cong \left \{ \begin{matrix} \ker \{ \Id - \widetilde T_{f,p}^{-n+1} \}, && \text{ if $k = -n+1$; } \\ 0, && \text{ if $k \neq -n+1$.} \end{matrix} \right . \\
H^{-n+1}(\im \{\Id-\overset{\thicksim}{T_f}\})_p &\cong \im \{\Id-\widetilde T_{f,p}^{-n+1}\}, \\
H^{-n+1}(\coker\{\Id-\overset{\thicksim}{T_f}\})_p &\cong \coker\{\Id-\widetilde T_{f,p}^{-n+1}\},
\end{align*}
where $\Id-\widetilde T_{f,p}^{-n+1}$ is the Milnor monodromy action on $H^1(F_{f,p};\Q)$.
\end{prop}
\begin{proof}
Since $H^k(\ker \{\Id - \widetilde T_f\})_p = 0$ for $k \neq -n+1$, the result follows from the short exact sequences
$$
0 \to H^{-n+1}(\ker\{\Id-\overset{\thicksim}{T_f}\})_p \to H^1(F_{f,p};\Q) \to H^{-n+1}(\im \{\Id-\overset{\thicksim}{T_f}\})_p \to 0,
$$
and 
$$
0 \to H^{-n+1}(\im \{\Id-\overset{\thicksim}{T_f}\})_p \to H^1(F_{f,p};\Q) \to H^{-n+1}(\coker\{\Id-\overset{\thicksim}{T_f}\})_p \to 0.
$$
\end{proof}

By taking stalk cohomology of the fundamental short exact sequence, we have
$$
0 \to H^{-n}(\Q_X^\bullet[n])_p \to H^{-n}(\Idot_X)_p \to \ker \{ \Id-  \widetilde T_{f,p}^{-n+1}\} \to 0.
$$
Since $\pi_*\Idot_Y \cong \Idot_X$, and $H^{-n}(\pi_*\Idot_Y)_p \cong \Q^{|\pi^{-1}(p)|}$, 
$$
\ker \{ \Id-  \widetilde T_{f,p}^{-n+1}\} \cong \Q^{|\pi^{-1}(p)|-1}
$$
for all $p \in X$, yielding the following nice lower-bound:

\begin{cor}\label{cor:milnorlowerbound}
$$
\dim_\Q H^1(F_{f,p};\Q) \geq |\pi^{-1}(p)|-1.
$$
\end{cor}

\bigskip

\begin{rem}
In the case where $X = V(f)$ is a hypersurface with smooth normalization in some open neighborhood $\U$ of the origin in $\C^{n+1}$, we prove in \cite{hepdefhyper} that a strong relationship holds between the \textbf{characteristic polar multiplicities} of $\Ndot_X$ and the L\^{e} numbers of the function $f$ (Theorem 5.2). 

\smallskip

A careful observation yields that the same result holds for hypersurfaces whose normalizations are $\Q$-homology manifolds (since all computations in the Theorem take place inside the hypersurface $V(f)$). Moreover, one can even use the same proof as Theorem 5.2 to obtain this more general result. 




\end{rem}

\begin{rem}\label{rem:saitomonodromy}
 In the hypersurface case $X = V(f)$, Saito's calculation of $H^0(X;\mathcal{F}_X)$ via invariant cycles of the monodromy (\cite{msaitow0} Section 2.4) is especially interesting to us. 

Suppose the normalization of $X$ is a rational homology manifold. Massey's result that $\Ndot_X \cong \ker\{\Id-\widetilde T_f\}$ together with \propref{prop:identities} allows us to identify $\mathcal{F}_X$ with the constructible sheaf ${\mathbf {ker}} \{\Id - \widetilde T_f^{-n+1} \}$ whose stalk at a point $p$ is 
$$
\ker \{ \Id - \widetilde T_{f,p}^{-n+1}\} \subseteq H^{-n+1}(\phi_f[-1]\Q_\U^\bullet[n+1])_p \cong H^1(F_{f,p};\Q),
$$
where $\widetilde T_{f,p}^{-n+1}$ is the Milnor monodromy operator on $H^1(F_{f,p};\Q)$. Consequently, Saito's calculation of $H^0(X;\mathcal{F}_X)$ via the internal monodromy of $\mathcal{F}_X$ allows us to compute information about the Milnor monodromy of $f$.
\end{rem}

\bigskip

\section{Example}\label{sec:examples}

We consider the following ``trivial, non-trivial" example of the normalization of a surface $X$ with one-dimensional singularity in $\C^3$, which nicely illustrates the content of \thmref{thm:main}. 
\bigskip

Let $f(x,y,z) = xz^2-y^2(y+x^3)$, so that $X = V(f) \subseteq \C^3$ has critical locus $\Sigma f = V(y,z)$. Then, if we let $Y = V(u^2-x(y+x^3),uy-xz,uz-y(y+x^3)) \subseteq \C^4$, the projection map $\pi : Y \to X$ is the normalization of $X$.

\smallskip

It is easy to check that $\Sigma Y= V(x,y,z,u)$, and 
$$
\pi^{-1}(\Sigma f) = V(u^2-x^{4},y,z).
$$
It then follows that $X_k = \emptyset$ if $k > 2$, and $X_2 = V(y,z) \backslash \{\0\}$, so that 
$$
\supp \Ndot_X = V(y,z) = \Sigma f.
$$

\medskip

For $p \in X$,  
\begin{align*}
H^{-2}(\pi_*\Idot_Y)_p \cong \bigoplus_{q \in \pi^{-1}(p)} H^{-2}(\Idot_Y)_q  \quad (\dag 4.1)\\
\end{align*}
But $\pi^{-1}(p) \subseteq Y \backslash \Sigma Y$, and $\left (\Idot_Y \right )_{|_{\pi^{-1}(p)}} \cong \left (\Q_Y^\bullet[2] \right )_{|_{\pi^{-1}(p)}}$, so from $(\dag 4.1)$, it follows that
$$
H^{-2}(\pi_*\Idot_Y)_p \cong \Q^2.
$$
Similarly, since $\left (\Idot_Y \right )_{Y \backslash \Sigma Y} \cong \Q_{Y \backslash \Sigma Y}^\bullet[2]$, it follows that 
$$
H^0(\Ndot_X)_p \cong H^{-1}(\pi_*\Idot_Y)_p = 0.
$$

\medskip

When $p = \0$, we find
$$
H^k(\Idot_Y)_\0 \cong \left \{ \begin{matrix} \HH^k(K_{Y,\0};\Idot_Y), && \text{ if $k \leq -1$} \\ 0, && \text{ if $k > -1$} \end{matrix} \right .
$$
Since $Y$ has an isolated singularity at the origin in $\C^4$, we further have
$$
\HH^k(K_{Y,\0};\Idot_Y) \cong H^{k+2}(K_{Y,\0},\Q).
$$
For $0 < \epsilon \ll 1$, the sphere $S_\epsilon$ transversely intersects $Y$ near $\0$, so the real link $K_{Y,\0} = Y \cap S_\epsilon$ is a compact, orientable, smooth manifold of (real) dimension 3. We are interested in computing the two integral cohomology groups $H^0(K_{Y,\0};\Q)$ and $H^1(K_{Y,\0};\Q)$. 

Because $K_{Y,\0}$ is also connected, we can apply Poincar\'e duality to find $H^0(K_{Y,\0};\Q) \cong \Q$. 

Consider the standard parameterization of the twisted cubic $\nu : \P^1 \to \P^3$ via
$$
\nu([s:t]) = [s^3:st^2:t^3:s^2t] = [x:y:z:u]
$$
which lifts to a map $\nu : \C^2 \to \C^4$, parameterizing the affine cone over the twisted cubic, i.e., the normalization $Y = V(u^2-xy,uy-xz,uz-y^2)$. Then, we claim that $\nu$ is a 3-to-1 covering map away from the origin. Clearly, since $\nu$ parameterizes $Y$, we see that $\nu$ is a surjective local diffeomorphism onto $\nu(\C^2) = Y$. 

Suppose that $\nu(s,t) = \nu(s',t')$. Then, we must have $s^3 = (s')^3$ and $t^3 = (t')^3$, so that there are cube roots of unity $\eta$ and $\omega$ for which $s = \eta s'$ and $t = \omega t'$. But then,
$$
s^2t = (s')^2(t') = \eta^2 \omega s^2 t,
$$
so either $\eta^2 \omega = 1$, or $st = 0$. Since $\eta$ and $\omega$ are both cube roots of unity, if $\eta^2 \omega =1$, then $\eta = \omega$. Additionally, note that $st = 0$ implies $(s,t) = \0$. It then follows that $\nu$ is 3-to-1 away from the origin. 

Consider then the (real analytic) function 
$$
r(x,y,z,u) = |x|^2 + 3|y|^2 + |z|^2 + 3|u|^2
$$
on $\C^4$; $r$ is proper, transversally intersects $Y$ away from $\0$, and $Y \cap r^{-1}[0,\epsilon)$ gives a fundamental system of neighborhoods of the origin in $Y$. Consequently, $Y \cap r^{-1}(\epsilon)$ gives, up to homotopy, the real link $K_{Y,\0}$. The composition $r(\nu(s,t))$ then gives:
\begin{align*}
r(\nu(s,t)) &= |s^3|^2 + 3|st^2|+|t^3|^2 + 3|s^2t|^2 \\
&= |s|^6 + 3|s|^4|t|^2 +3|s|^2|t|^2+|t|^6 \\
&= \left (|s|^2+|t|^2 \right )^3 = \epsilon,
\end{align*}
provided that $|s|^2 + |t|^2 = \sqrt[3]{\epsilon}$; that is, $\nu$ maps the 3-sphere in $\C^2$ 3-to-1 onto the real link $K_{Y,\0}$. Since the 3-sphere is simply-connected, it is the universal cover of $K_{Y,\0}$. The group of deck transformations given by multiplying $(s,t)$ by a cube root of unity then yields the isomorphism $\pi_1(K_{Y\0}) \cong \Z/3\Z$. Thus, $H_1(K_{Y,\0};\Z) \cong \Z/3\Z$. 

By again applying Poincar\'e duality, we find $H^2(K_{Y,\0};\Z) \cong \Z/3\Z$ as well. By the Universal Coefficient theorem for cohomology, we then have $H_2(K_{Y,\0};\Z) = 0$ so that $H^1(K_{Y,\0};\Z) = 0$ by Poincar\'e duality. Using $\Q$ coefficients, this implies:
$$
H^k(K_{Y,p};\Q) \cong \left \{ \begin{matrix} \Q, && \text{ if $k= 0,3$} \\ 0, && \text{ else } \end{matrix} \right .
$$
for all $p \in Y$, so that $Y$ is a $\Q$-homology manifold.

Equivalently, we find: 
\begin{align*}
H^k(\Ndot_X)_p \cong \left \{ \begin{matrix} \Q, && \text{ if $k=-1$ and $p \in \Sigma f \backslash \{\0\}$} \\ 0, && \text{ if $k \neq -1$, $p \in \Sigma f$} \end{matrix} \right . 
\end{align*}
i.e., $\Ndot_X$ has stalk cohomology concentrated in degree $-1$.

It is not hard to show that the monodromy of the local system $H^{-1}(\Ndot_X)_{|_{\Sigma f \backslash \{\0\}}}$ is trivial; consequently, ${\Ndot_X}_{|_{\Sigma f}}$ is isomorphic to the extension by zero of the constant sheaf on $\Sigma f \backslash \{\0\}$. That is, if $j: \Sigma f \backslash \{\0\} \hookrightarrow \Sigma f$ is the open inclusion, then ${\Ndot_X}_{|_{\Sigma f}} \cong j_! \Q_{\Sigma f \backslash \{\0\}}^\bullet[1]$. In particular, we see that $\Ndot_X$ is not semi-simple as a perverse sheaf on $X$. 

To compare with \remref{rem:semisimple}, this failure to be a semi-simple perverse sheaf can be detected by the presence of $W_0 \Ndot_X \cong \Q_{\{\0\}}^\bullet \neq \0$.

\printbibliography

\end{document}